\documentclass[reqno]{amsart}
\usepackage{amssymb,amsthm, amsmath, amsfonts, amsbsy}
\usepackage{graphics}
\usepackage{hyperref}
\usepackage[all]{xy}
\usepackage{enumerate}
\usepackage[mathscr]{eucal}
\usepackage{bbm}

\theoremstyle{plain}
\newtheorem{theorem}{Theorem}[section]
\newtheorem{lemma}[theorem]{Lemma}

\newtheorem{corollary}[theorem]{Corollary}
\numberwithin{equation}{section}

\theoremstyle{definition}

\newtheorem{definition}[theorem]{Definition}
\newtheorem{example}[theorem]{Example}
\newtheorem{remark}[theorem]{Remark}

\DeclareMathOperator{\reg}{reg}
\DeclareMathOperator{\Mod}{-Mod}

\DeclareMathOperator{\Tor}{Tor}
\DeclareMathOperator{\hd}{hd}
\DeclareMathOperator{\gd}{gd}
\DeclareMathOperator{\prd}{prd}
\DeclareMathOperator{\im}{im}

\DeclareMathOperator{\coker}{coker}

\newcommand{\C}{{\mathscr{C}}}
\newcommand{\bfr}{{\mathbf{r}}}
\newcommand{\bft}{{\mathbf{t}}}
\newcommand{\bfn}{{\mathbf{n}}}
\newcommand{\bff}{{\mathbf{f}}}
\newcommand{\bfg}{{\mathbf{g}}}
\newcommand{\bfh}{{\mathbf{h}}}
\newcommand{\bfo}{{\mathbf{1}}}

\newcommand{\FI}{{\mathrm{FI}}}

\newcommand{\N}{{\mathbb{N}}}
\newcommand{\Ob}{{\mathrm{Ob}}}

\title{Castelnuovo-Mumford regularity of representations of certain product categories}

\author{Wee Liang Gan}
\address{Department of Mathematics, University of California, Riverside, CA 92521, USA}
\email{wlgan@ucr.edu}

\author{Liping Li}
\address{LCSM(Ministry of Education), School of Mathematics and Statistics, Hunan Normal University, Changsha, Hunan 410081, China.}
\email{lipingli@hunnu.edu.cn}

\thanks{The second author was supported by the National Natural Science Foundation of China (Grant No. 11771135), HuXiang High-Level Talents Gathering Project by the Science and Technology Department of Hunan Province (Grant No. 2019RS1039), and the Research Foundation of Education Bureau of Hunan Province (Grant No. 18A016). Both authors appreciate the anonymous referee for carefully checking the manuscript and providing many very helpful and insightful comments.}

\begin{document}

\begin{abstract}
We show in this paper that representations of a finite product of categories satisfying certain combinatorial conditions have finite Castelnuovo-Mumford regularity if and only if they are presented in finite degrees, and hence the category consisting of them is abelian. These results apply to examples such as the categories $\FI^m$ and $\FI_G^m$.
\end{abstract}

\maketitle

\section{Introduction}

Representation theory of infinite categories is a relatively new research area with applications in topology, geometric group theory, algebraic geometry, and commutative algebras; see for instance \cite{CE, CEF, CEFN, Gad}. To study representation theoretic properties of an infinite category $\C$, usually the first step is to find a suitable abelian category of representations in which homological algebra can be carried out. Since the category of all representations of $\C$ (also called $\C$-modules) is too large for practical purpose, one naturally considers the category of finitely generated $\C$-modules. However, this idea does not always work well because of the following reasons: (1), there are many infinite categories (for example, the poset of positive integers equipped a partial ordering induced by division) which are not \emph{locally Noetherian} even when the coefficient ring is a field;, that is, submodules of finitely generated $\C$-modules might not be finitely generated; (2), it is not an easy task to show the locally Noetherian property of $\C$ over a commutative Noetherian coefficient ring, see for instance \cite{SS}; (3), for some applications in topology, people frequently have to consider infinitely generated $\C$-modules over arbitrary coefficient rings. Therefore, in many cases we expect to find intermediate abelian subcategories between the category of all $\C$-modules and the category of finitely generated $\C$-modules.

A significant resolution of this question was established by Church and Ellenberg in \cite{CE} for the category $\FI$ of finite sets and injections. They proved that if a representation $V$ of $\FI$ over an arbitrary commutative coefficient ring is \emph{presented in finite degrees} (see Definition \ref{presented in finite degrees}), then its \emph{Castelnuovo-Mumford regularity} (or \emph{regularity} for short, see Definition \ref{homological degrees}) is finite. Furthermore, they obtained an explicit and simple upper bound for the regularity in \cite[Theorem A]{CE}. This result immediately implies that representations of $\FI$ presented in finite degrees form an abelian category.

The main goal of this paper is to extend the above result of Church and Ellenberg to finite products $\C_1\times \cdots \times \C_m$ of categories $\C_1, \ldots, \C_m$ satisfying certain combinatorial conditions (see Subsections 2.2 and 3.1). An example of such finite products is the category $\FI^m$, the product of $m$ copies of $\FI$ for every positive integer $m$, which was introduced by Gadish in \cite{Gad, Gad2} and whose representation theoretic properties were also studied in \cite{LR2, LY2}.
Let $\C=\C_1\times \cdots \times \C_m$.
Under the combinatorial assumptions, objects of $\C$ form a ranked poset, so its representations (or $\C$-modules, see the definition in Subsection 2.3) have a graded structure, and hence we can define regularity for them. Based on an inductive machinery developed by the authors in \cite{GL3} as well as a key observation that a numerical invariant of an $\C$-module $V$ is finite whenever $V$ is presented in finite degrees, we prove the following result, partially answered a question proposed in \cite[Subsection 6.1]{LY2}.

\begin{theorem} \label{theorem for FI}
Let $\C=\C_1\times \cdots \times \C_m$ where $\C_1, \ldots, \C_m$ are categories satisfying the combinatorial conditions specified in Subsections 2.2 and 3.1. Then a $\C$-module $V$ over an arbitrary commutative coefficient ring has finite Castelnuovo-Mumford regularity if and only if it is presented in finite degrees.
\end{theorem}

The basic strategy to prove the above theorem is as follows. The category $\C$ is equipped with $m$ \emph{self-embedding functors}, which induce $m$ \emph{shift functors} $\Sigma_i$, $i \in [m]$, in the module category (see Subsection 3.1). When $V$ is a \emph{torsion-free} $\C$-module (see Subsection 2.6), there is a short exact sequence of $\C$-modules:
\begin{equation*}
0 \to V^{\oplus m} \to \bigoplus_{i \in [m]} \Sigma_i V \to \bigoplus_{i \in [m]} D_iV \to 0,
\end{equation*}
where $D_i$'s are the cokernel functors, and one can deduce the finiteness of regularity of $V$ from that of the third term via induction. For an arbitrary $\C$-module $V$, we construct a tree $\mathcal{T}(V)$ of quotient modules of $V$ recursively. We prove that it is a finite tree when $V$ is presented in finite degrees, and furthermore the modules on the lowest level are torsion free, turning to the special case previously handled. We also prove the fact that if all modules on a level of $\mathcal{T}(V)$ have finite regularity, so are all members lying on the level above it. This recursive method allows us to verify the finiteness of regularity of all members in $\mathcal{T}(V)$, in particular that of $V$.

A cornerstone of the above strategy is the finiteness of the tree $\mathcal{T}(V)$. In \cite{GL3}, the authors considered finitely generated representations of several infinite combinatorial categories over commutative Noetherian rings, where the finiteness of this tree follows directly from the locally Noetherian property of these categories. In this paper we introduce a new idea; that is, for $\C$-modules $V$ presented in finite degrees, we assign an integral numerical invariant to each member in $\mathcal{T}(V)$, and show that if a member is a \emph{child} (see Definition \ref{children}) of another member, then the numerical invariant of the former is strictly smaller than that of the later one. Now the finiteness of $\mathcal{T}(V)$ is guaranteed by the fact that the numerical invariants of all $\C$-modules share a common lower bound.

This paper is organized as follows. In Section \ref{preliminaries}, we describe the setting for our results and recall some basic definitions. In Section \ref{shift functors}, we discuss shift functors. In Section \ref{main proof}, for every $\C$-module, we construct a tree of quotient modules and use it to give a proof for the main theorem. The last section contains corollaries of the main theorem.

\section{Preliminaries} \label{preliminaries}

\subsection{Notations}
Denote by $\N$ the set of non-negative integers. For each $n\in \N$, we set $[n]=\{1,\ldots, n\}$; in particular, $[0]=\emptyset$. We fix a commutative ring $k$ with identity.

\subsection{The categories $\C_1, \ldots, \C_m$ and their product $\C$.}
Throughout this paper, we let $\C$ be the product $\C_1\times \cdots \times \C_m$ where each $\C_i$ is a category satisfying the following conditions:
\begin{itemize}
\item
$\Ob(\C_i)=\N$.

\item
For each $r,n\in \Ob(\C_i)$, the morphism set $\C_i(r, n)$ is nonempty if and only if $r\leqslant n$.

\item
If $r<m<n$, every morphism $f: r \to n$ in $\C_i$ is a composition $f=hg$ for some morphisms $g:r\to m$
and $h:m\to n$.

\item
For each $n\in\Ob(\C_i)$, every endomorphism of $n$ in $\C_i$ is an isomorphism.

\item
If $r<n$, the group $\C_i(n,n)$ acts transitively on the set $\C_i(r,n)$.
\end{itemize}
In the terminology of \cite{GL1}, $\C_i$ is an EI category of type A$_\infty$ satisfying the transitivity condition.
\begin{example}
The skeleton of $\FI$ whose objects are $[n]$ for all $n\in \N$ satisfies the above conditions.
\end{example}

The objects of the category $\C=\C_1\times \cdots\times \C_m$ are $m$-tuples $\bfn = (n_1, \ldots, n_m)$ where each $n_i\in \N$, and the morphisms are $m$-tuples $\bff = (f_1, \ldots, f_m): \bfr \to \bfn$ where each $f_i\in \C_i(r_i, n_i)$.
There is a partial ordering on $\Ob(\C)$ defined as follows: $\bfr \leqslant \bfn$ if and only if $r_i \leqslant n_i$ for each $i\in [m]$, or equivalently, the morphism set $\C (\bfr, \bfn)$ is nonempty. We note that $\C$ is a \emph{graded category}. Explicitly, for an object $\bfn$ and a morphism $\bff: \bfr \to \bfn$, one defines the rank $|\bfn| = \sum_{i \in [m]} n_i$ and the degree $|\bff| = |\bfn| - |\bfr|$. It is clear that every morphism of degree $s$ in $\C$ can be written as a composition of $s$ morphisms of degree 1.

\subsection{$\C$-modules}
By definition, a \textit{$\C$-module} (or a representation of $\C$) is a covariant functor from $\C$ to $k \Mod$, the category of all $k$-modules, and a homomorphism between two $\C$-modules is a natural transformation. Since $k \Mod$ is an abelian category, so is $\C \Mod$, the category of all $\C$-modules. Furthermore, the abelian category $\C\Mod$ has enough projective objects (see, for example, \cite[Exercise 2.3.8]{Wei}). In particular, the $k$-linearization of representable functors $M(\bfn) = k\C(\bfn, -)$ are projective $\C$-modules. By a \emph{free $\C$-module}, we mean a $\C$-module isomorphic to one of the form $\bigoplus_{i\in I} M(\bfn_i)$ where $I$ is any indexing set and $\bfn_i \in \Ob(\C)$. For every $\C$-module $V$, there is a surjective $\C$-module homomorphism
\begin{equation} \label{surject}
\bigoplus _{\bfn \in \mathbb{N}^m} M(\bfn)^{\oplus c_{\bfn}} \to V \to 0
\end{equation}
where the multiplicities $c_{\bfn}$ could be infinity. It follows that every projective $\C$-module is a direct summand of a free $\C$-module.

\begin{definition} \label{presented in finite degrees}
A $\C$-module $V$ is \emph{generated in finite degrees} if there exists a natural number $N_V$, dependent on $V$, and a surjective homomorphism of the form \eqref{surject} such that the multiplicity $c_{\bfn}$ can be taken to be 0 for objects $\bfn$ satisfying $|\bfn| \geqslant N_V$. The $\C$-module $V$ is \emph{presented in finite degrees} if there exists a surjective homomorphism of the above form such that both $V$ and the kernel are generated in finite degrees.
\end{definition}

\begin{remark} \normalfont
The above definition can be restated in a more intuitive way. That is, a $\C$-module $V$ is generated in finite degrees if there exists a natural number $N_V$ such that for every object $\bfn$ with $|\bfn| \geqslant N_V$, the value $V_{\bfn}$ of $V$ on the object $\bfn$ satisfies the following equality:
\begin{equation*}
V_{\bfn} = \sum_{\begin{matrix} \bfr \lneqq \bfn \\ \bff \in \C(\bfr, \bfn) \end{matrix}} \bff \cdot V_{\bfr},
\end{equation*}
where $\bff \cdot V_{\bfr}$ means $V(\bff) (V_{\bfr})$, the image of $V_{\bfr}$ under the $k$-linear map $V(\bff)$.
\end{remark}

\subsection{The category algebra.} There is another way to study $\C$-modules from the traditional ring theoretic viewpoint. The \emph{category algebra} $k\C$ is set to be the free $k$-module whose basis is the set of morphisms in $\C$. Multiplication is defined by the following rule: given two morphisms $\bff: \bfr \to \bfn$ and $\bfg: \textbf{q} \to \textbf{t}$, let $\bff \cdot \bfg$ be the composition (from right to left) $\bff \circ \bfg$ if $\textbf{t} = \bfr$, and 0 otherwise. Then $k\C$ is a non-unital associative $k$-algebra. Furthermore, the graded structure of $\C$ induces a graded structure on $k\C$, and it is generated in degrees 0 and 1. Explicitly, one has the following decomposition:
\begin{equation*}
k\C = \bigoplus_{s \in \mathbb{N}} \Big{(} \bigoplus _{\begin{matrix} \bfn, \bfr \in \mathbb{N}^m \\ |\bfn| - |\bfr| = s \end{matrix}}  k\C(\bfr, \bfn) \Big{)}.
\end{equation*}
In particular, the direct sum $\mathfrak{m}$ over all $s\geqslant 1$ is a two-sided ideal of $k\C$; it is precisely the free $k$-module spanned by all non-invertible morphisms.

Let $k\C \Mod$ be the category of all $k\C$-modules. By \cite[Theorem 7.1]{M}, the category $\C \Mod$ may be identified with a full subcategory of $k\C \Mod$. In particular, the free $\C$-modules $k\C(\bfn, -)$ are identified with projective $k\C$-modules of the form $k\C e_{\bfn}$, where $e_{\bfn}$ is the identity morphism on $\bfn$ and is also viewed as an idempotent in the algebra $k\C$.

\subsection{Homology groups of $\C$-modules.} Let $V$ be a $\C$-module. The \emph{zeroth homology group} $H_0(V)$ of $V$ is defined as follows. For each $\bfn \in \Ob(\C)$, let
\begin{equation*}
(H_0(V))_{\bfn} = V_{\bfn} \bigg/ \sum_{\bfr \lneqq \bfn} k\C(\bfr, \bfn) \cdot V_{\bfr}.
\end{equation*}
Then $H_0(V)$ is a $\C$-module on which every non-invertible morphism in $\C$ acts as the zero map.

The functor $H_0: V \mapsto H_0(V)$ is right exact, so we make the following definition:

\begin{definition} \label{homological degrees}
For each $s \in \mathbb{N}$, we define $H_s:\C\Mod\to\C\Mod$ to be the $s$-th left derived functor of $H_0$. For any $\C$-module $V$, the $s$-th homological degree $\hd_s(V)$ is defined by
\begin{equation*}
\hd_s(V) = \sup \{ |\bfn| \mid (H_s(V))_{\bfn} \neq 0\}
\end{equation*}
or $-1$ when the above set is empty. The Castelnuovo-Mumford regularity of $V$ is:
\begin{equation*}
\reg(V) = \sup \{ \hd_s(V) - s \mid s \in \mathbb{N} \}.
\end{equation*}
\end{definition}

Suppose $V$ is a $\C$-module. We call $\hd_0(V)$ and $\max\{\hd_0(V), \hd_1(V)\}$ the \emph{generation degree} and \emph{presentation degree} of $V$, and denote them by $\gd(V)$ and $\prd(V)$ respectively. It is easy to see that there exists a surjective morphism $P\to V$ such that $P$ is a free $\C$-module with $\gd(P)=\gd(V)$.

\begin{remark} \normalfont
In \cite{GL3, LY2}, the authors provided another abstract version for the above definitions. Recall that the graded category algebra $k\C$ has a two-sided ideal $\mathfrak{m}$ spanned by all non-invertible morphisms in $\C$. Then one can check that $H_0(V) \cong (k\C/\mathfrak{m}) \otimes_{k\C} V$, and hence $H_s(V) \cong \Tor_s^{k\C} (k\C/\mathfrak{m}, V)$. Moreover, using the language of homological degrees, Definition \ref{presented in finite degrees} can be restated as follows: a $k\C$-module $V$ is generated in finite degrees if $\gd(V)$ is finite, and is presented in finite degrees if $\prd(V)$ is finite.
\end{remark}

\subsection{Torsion theory of $\C$-modules.}
Let $V$ be a $\C$-module. An element $v \in V_{\bfr}$ for some $\bfr\in\Ob(\C)$ is \emph{torsion} if for some $\bfn\in\Ob(\C)$ and $\bff \in \C(\bfr, \bfn)$, one has $\bff \cdot v = 0$. Note that the group $\C(\bfn,\bfn)$ acts transitively on $\C(\bfr, \bfn)$, so if $\bff \cdot v = 0$ for some $\bff \in \C(\bfr, \bfn)$, then all morphisms in $\C(\bfr, \bfn)$ send $v$ to 0. We say that $V$ is a \emph{torsion module} if for every $\bfn\in \Ob(\C)$, all elements $v \in V_{\bfn}$ are torsion. For an arbitrary $\C$-module $V$, there is a canonical short exact sequence $0 \to V_T \to V \to V_F \to 0$ such that $V_T$ is torsion, and $V_F$ is \emph{torsion-free}; that is, $V_F$ contains no nonzero torsion elements.

We now define a key invariant $t(V)$ recording certain information of the torsion part of $V$.

\begin{definition} \label{torsion vector}
Suppose $V$ is a $\C$-module.  For each $i\in [m]$, define
\begin{equation*}
t_i(V) = \sup \{ s \mid \exists \, \bfn \text{ and } 0 \neq v \in V_{\bfn} \text{ such that } k\C(\bfn, \bfn + \bfo_i) \cdot v = 0 \text{ and } n_i = s \},
\end{equation*}
where $\bfo_i$ is the $m$-tuple with 1 at the $i$-th coordinate and zeroes at other coordinates; if the above set is empty, we set $t_i(V) = -1$ by convention. We define the torsion vector $\bft (V)$ of $V$ to be the $m$-tuple $(t_1(V), \ldots, t_m(V))$, and define
\[ t(V) = t_1(V) + \cdots +  t_m(V). \]
\end{definition}

It is easy to see that $t(V)$ has a lower bound $-m$, and $t(V) = -m$ if and only if $V$ is torsion-free. In general, $t(V)$ might not be a finite number.

\begin{remark} \normalfont
Loosely speaking, we can view objects in $\C$ as ``integral points" in the free module $k^m$. For $i \in [m]$, if $t_i(V) \geqslant 0$ (including the case that $t_i(V)$ is infinite), then there exist a ``hyperplane" perpendicular to the $i$-th coordinate axis, an object $\bfn$ lying in this hyperplane, and a nonzero element $v \in V_{\bfn}$ such that $v$ becomes 0 when it moves one step along the $i$-th direction. In this situation, $t_i(V)$ is precisely the supremum of the ``heights" of these hyperplanes.
(The torsion vector of an $\FI^m$-module $V$ was introduced in \cite[Definition 2.8]{LY2} with a different version of definition. We would like to point out that $t(V)$ is distinct from the \emph{torsion degree} in \cite[Definition 2.8]{LY2}, which is not the sum but the supremum of these $t_i(V)$'s.)
\end{remark}

\section{Shift functors} \label{shift functors}

\subsection{Assumptions and definitions.}
From now on, we make the following assumptions on the category $\C_i$ for every $i\in [m]$:
\begin{enumerate}[(i)]
\item
We assume that there is a faithful functor
\[ \psi_i : \C_i\to\C_i \]
such that $\psi_i(n)=n+1$ for every $n\in \Ob(\C_i)$. Define the \emph{shift functor} $\Psi_i$ on $\C_i\Mod$ to be the pullback functor $\psi_i^*$, that is,
\[ \Psi_i: \C_i\Mod\to  \C_i\Mod, \qquad V \mapsto V\circ \psi_i. \]
Note that for every $\C_i$-module $V$ and $n\in\Ob(\C_i)$, one has $(\Psi_i V)_n = V_{n+1}$.

\item
We assume that there is a natural transformation
\[ \varepsilon_i: \mathrm{Id}_{\C_i} \to \psi_i, \]
where $\mathrm{Id}_{\C_i}$ denotes the identity functor on $\C_i$. Note that for each $n\in\Ob(\C_i)$, one has the morphism $(\varepsilon_i)_n\in \C_i(n, n+1)$. The natural transformation $\varepsilon_i$ induces a natural transformation
\[ \mu_i : \mathrm{Id}_{\C_i\Mod} \to \Psi_i, \]
where $\mathrm{Id}_{\C_i\Mod}$ denotes the identity functor on $\C_i\Mod$. Explicitly, for each $\C_i$-module $V$, the homomorphism
\[(\mu_i)_V: V \to \Psi_i V \]
is defined at each $n\in\Ob(\C_i)$  by $V_n \to V_{n+1}$, $v\mapsto (\varepsilon_i)_n \cdot v$.

\item
For each $n\in \Ob(\C_i)$, denote the free $\C_i$-module $k\C_i(n,-)$ by $M_i(n)$.
We assume that the homomorphism
\[ (\mu_i)_{M_i(n)} : M_i(n) \to \Psi_i M_i(n) \]
is injective and its cokernel is a projective $\C_i$-module with generation degree $\leqslant n-1$.
\end{enumerate}

\begin{remark}
In the terminology of \cite[Definition 1.2]{GL2}, the shift functor $\Psi_i$ is, in particular, a generic shift functor.
\end{remark}

For any $\C_i$-module $V$, we denote by $\Delta_i V$ the cokernel of $(\mu_i)_V: V \to \Psi_i V$. Thus, for each $n\in \N$, $\Delta_i M_i(n)$ is a projective $\C_i$-module and $\gd(M_i(n))\leqslant n-1$.

\begin{lemma} \label{shift of free module}
Let $i\in [m]$. For every $n\in \Ob(\C_i)$, the $\C_i$-module $\Psi_i M_i(n)$ is projective and has generation degree $\leqslant n$.
\end{lemma}
\begin{proof}
There is a short exact sequence
\[ 0 \to M_i(n) \to \Psi_i M_i(n) \to \Delta_i M_i(n)  \to 0. \]
Since $M_i(n)$ and $\Delta_i M_i(n)$ are projective $\C_i$-modules, it follows that $\Psi_i M_i(n)$ is also projective. Since $\gd (M_i(n)) = n$ and $\gd (\Delta_i M_i(n)) \leqslant n-1$, we have $\gd (\Psi_i M_i(n))\leqslant n$.
\end{proof}

\begin{remark} \label{FI has shift}
If $\C_i$ is the skeleton of $\FI$ whose objects are $[n]$ for all $n\in \N$, then $\C_i$ satisfies all the above assumptions (see \cite[Proposition 2.12]{CEFN}).
\end{remark}

\subsection{Shift functors for $\C$}
Recall that $\C=\C_1\times \cdots \times \C_m$. Let $i \in [m]$. We define the functor $\rho_i: \C \to \C$ by
\[ \rho_i = (\mathrm{Id}_{\C_1}, \ldots, \psi_i, \ldots, \mathrm{Id}_{\C_m}) \]
where $\psi_i$ is in the $i$-th coordinate. Note that $\rho_i(\bfn) = \bfn + \bfo_i$ for each $\bfn\in\Ob(\C)$. We define the $i$-th \emph{shift functor} $\Sigma_i$ on  $\C\Mod$ to be the pullback functor $\rho_i^*$, that is,
\[ \Sigma_i: \C \Mod \to \C \Mod, \quad V \mapsto V \circ \rho_i.\]

Similarly, the natural transformation $\varepsilon_i: \mathrm{Id}_{\C_i} \to \psi_i$ induces a natural transformation $\mathrm{Id}_{\C} \to \rho_i$, which in turn induces a natural transformation $\mathrm{Id}_{\C\Mod} \to \Sigma_i$. Thus, for each $\C$-module $V$, we have a natural homomorphism $V \to \Sigma_i V$. Explicitly, for each $\bfn \in \Ob(\C)$, the natural map $V_\bfn \to (\Sigma_i V)_\bfn$ is defined by $v \mapsto \bfh \cdot v$ where $\bfh \in \C(\bfn, \bfn+\bfo_i)$ with $h_i = (\varepsilon_i)_{n_i} \in \C_i(n_i, n_i+1)$ and $h_j = \mathrm{Id}_{n_j} \in \C_j(n_j, n_j)$ if $j\neq i$.

Let $K_i V$ and $D_i V$ be, respectively, the kernel and cokernel of the natural homomorphism $V\to \Sigma_i V$, so that we have a natural exact sequence
\begin{equation*}
0 \to K_iV \to V \to \Sigma_i V \to D_i V \to 0.
\end{equation*}
Suppose $\bfn\in \Ob(\C)$. Since the group $\C(\bfn+\bfo_i, \bfn+\bfo_i)$ acts transitively on $\C(\bfn, \bfn+\bfo_i)$, it follows that
\begin{align*}
(K_i V)_\bfn &= \{ v\in V_\bfn \mid \bff \cdot v = 0 \mbox{ for some } \bff \in \C(\bfn, \bfn+\bfo_i)\} \\
&=   \{ v\in V_\bfn \mid \bff \cdot v = 0 \mbox{ for every } \bff \in \C(\bfn, \bfn+\bfo_i) \}.
\end{align*}

For each $i \in [m]$ and $\C$-module $V$, we denote by $V^i$ the $i$-th copy of $V$ in the direct sum $V^{\oplus m}$. For any subset $S \subset [m]$, let
\[ V^S = \bigoplus_{i\in S} V^i, \qquad
\Sigma_S V = \bigoplus_{i\in S} \Sigma_i V,  \qquad
K_S V = \bigoplus_{i\in S} K_i V, \qquad
D_S V = \bigoplus_{i\in S} D_i V. \]
In particular, if $S$ is the empty set, then these are the zero module. We have a natural exact sequence
\[ 0 \to K_S V \to V^S \to \Sigma_S V \to D_S V \to 0.
\]
It is plain that the functor $\Sigma_S$ is exact, and the functor $D_S$ is right exact.

\begin{lemma} \label{basic properties}
Suppose $V$ is a $\C$-module and $S$ is a subset of $[m]$.
\begin{enumerate}
\item For every $\bfn\in\Ob(\C)$, the $\C$-modules $\Sigma_S M(\bfn)$ and $D_SM(\bfn)$ are projective. Moreover, one has: $\gd(\Sigma_S M(\bfn)) \leqslant |\bfn|$, and $\gd(D_S M(\bfn))\leqslant |\bfn|-1$.
\item If $V$ is nonzero, then one has:
\[ \gd(D_{[m]} V) \leqslant \gd(\Sigma_{[m]} V) \leqslant \gd(V) = \gd(D_{[m]}V) + 1.\]
\item  If $V$ is nonzero, then one has:
\[ \prd(\Sigma_S V) \leqslant \prd(V), \qquad \prd(D_S V) \leqslant \prd(V)-1. \]
In particular, if $V$ is presented in finite degrees, then $\Sigma_S V$ and $D_S V$ are also presented in finite degrees.
\item The $\C$-module $K_{[m]} V$ is zero if and only if $V$ is torsion-free.
\item For every $i \in [m]$, one has: $t_i(V) \leqslant \gd(K_i V)$.
\item For every $i, j \in [m]$, one has $t_i(\Sigma_j V) \leqslant t_i(V)$. Moreover, $t_i(\Sigma_iV) < t_i(V)$ whenever $t_i(V) \in \mathbb{N}$.
\end{enumerate}
\end{lemma}

\begin{proof}
(1) For any $i\in[m]$, one has:
\begin{gather*}
\Sigma_i M(\bfn) = M_1(n_1) \boxtimes \cdots \boxtimes \Psi_i M_i(n_i)  \boxtimes \cdots \boxtimes M_m(n_m), \\
D_i M(\bfn) = M_1(n_1) \boxtimes \cdots \boxtimes \Delta_i M_i(n_i)  \boxtimes \cdots \boxtimes M_m(n_m).
\end{gather*}
By Lemma \ref{shift of free module}, $\Psi_i M_i(n_i)$ is a projective $\C_i$-module, so it is a direct summand of a free $\C_i$-module. It follows that $\Sigma_i M(\bfn)$ is a direct summand of a free $\C$-module. Also, by Lemma \ref{shift of free module}, one has $\gd(\Psi_i M_i(n_i))\leqslant n_i$, so $\gd(\Sigma_i M(\bfn)) \leqslant |\bfn|$. Similarly for $D_iM(\bfn)$.

(2) Since $D_{[m]} V$ is a quotient of $\Sigma_{[m]} V$, we have $\gd(D_{[m]}V)\leqslant \gd(\Sigma_{[m]} V)$.

Let $P\to V$ be a surjective morphism where $P$ is a nonzero free $\C$-module with $\gd(P)=\gd(V)$. Since $\Sigma_{[m]}$ is exact and $D_{[m]}$ is right exact, we have surjective morphisms $\Sigma_{[m]} P \to \Sigma_{[m]} V$  and $D_{[m]} P \to D_{[m]} V$. It follows, using (1), that
\begin{gather*}
\gd(\Sigma_{[m]} V) \leqslant \gd(\Sigma_{[m]} P) \leqslant \gd(P)=\gd(V),\\
\gd(D_{[m]} V) \leqslant \gd(D_{[m]} P) \leqslant \gd(P)-1 = \gd(V)-1.
\end{gather*}

It remains to prove that $\gd(D_{[m]} V) \geqslant \gd(V)-1$. Suppose there exists $\bfn\in\Ob(\C)$ such that $|\bfn|\geqslant 1$ and $(H_0(V))_\bfn \neq 0$. Let $V'$ be the $\C$-submodule of $V$ generated by $V_{\bft}$ for all $\bft \in\Ob(\C)$ such that $|\bft|<|\bfn|$. Then we have $(V/V')_\bft = 0$ if $|\bft|<|\bfn|$. Since $|\bfn|\geqslant 1$, there exists $i\in[m]$ such that $n_i\geqslant 1$; set $\mathbf{r} = \bfn-\bfo_i \in \Ob(\C)$. Since $(V/V')_\mathbf{r} = 0$ but $(\Sigma_i (V/V'))_{\mathbf{r}} = (V/V')_\bfn \neq 0$, we have $(D_i (V/V'))_\mathbf{r} \neq 0$, and so $(D_{[m]} (V/V'))_\mathbf{r} \neq 0$, hence $D_{[m]} (V/V')$ is a nonzero $\C$-module. But for every $\bft\in \Ob(\C)$ such that $|\bft|<|\mathbf{r}|$, we have $(\Sigma_{[m]} (V/V'))_\bft = 0$, so $(D_{[m]} (V/V'))_\bft = 0$. Hence,
\[ \gd(D_{[m]} V) \geqslant \gd(D_{[m]} (V/V')) \geqslant |\mathbf{r}| = |\bfn| -1 . \]
It follows that $\gd(D_{[m]} V) \geqslant \gd(V)-1$.

(3) There is a short exact sequence $0 \to W \to P \to V \to 0$ such that $P$ is a free $\C$-module with $\gd(P)=\gd(V)$. We have:
\[ \gd(W) \leqslant \max\{\gd(P), \hd_1(V)\} = \prd(V).\]
Applying the exact functor $\Sigma_S$, we obtain the short exact sequence $0 \to \Sigma_S W \to \Sigma_S P \to \Sigma_S V \to 0$. By (1), $\Sigma_S P$ is a projective $\C$-module with $\gd(\Sigma_S P) \leqslant \gd(P)=\gd(V)$. Using (2), we have:
\[ \gd(\Sigma_S V) \leqslant \gd(V), \qquad  \hd_1(\Sigma_S V) \leqslant \gd(\Sigma_S W) \leqslant \gd(W) \leqslant \prd(V). \]
Therefore, $\prd(\Sigma_S V) \leqslant \prd(V)$.

Similarly, applying the right exact functor $D_S$, we get an exact sequence $D_S W \to D_S P \to D_S V \to 0$. Using (2), we have:
\[ \gd(D_S V) \leqslant \gd(V)-1, \qquad \gd(D_S W) \leqslant \max\{-1, \gd(W)-1\} \leqslant \prd(V)-1. \]
Noting that the kernel of $D_S P \to D_S V \to 0$ is a quotient module of $D_S W$, we conclude that
\[ \hd_1(D_S V) \leqslant \gd(D_S W) \leqslant \prd(V)-1. \]
Consequently, $\prd(D_S V)\leqslant \prd(V)-1$.

(4) If $K_{[m]} V$ is nonzero, then $K_i V$ is nonzero for some $i\in [m]$, in which case $(K_i V)_\bfn \neq 0$ for some $\bfn\in \Ob(\C)$, and so $V_\bfn$ contains a nonzero torsion element.

Conversely, suppose $V$ is not torsion-free. Then there exists a morphism $\bff \in \C(\bfn,\bfr)$ of positive degree such that $\bff \cdot v = 0$ for some nonzero $v\in V_\bfn$. Since every morphism is a composition of morphisms of degree $1$, there exists such a morphism $\bff$ of degree $1$, say $\bff\in \C(\bfn, \bfn+\bfo_i)$. It follows that $(K_i V)_\bfn \neq 0$, so $K_i V$ is nonzero, and hence $K_{[m]}V$ is nonzero.

(5) We only need to consider the case that $\gd(K_iV)$ is finite. An element $v \in V_{\bfn}$ is contained in $K_iV$ if and only if it vanishes when moving one step along the $i$-th direction. Thus $K_iV$ is either 0 or can be decomposed into a direct sum of direct summands, each of which is supported on a “hyperplane” perpendicular to the $i$-th axis.
If $t_i(V) > \gd(K_iV)$, then $t_i(V) \geqslant 0$. By Definition \ref{torsion vector}, there exist an object $\bfn$ and a nonzero element $v \in V_{\bfn}$ such that $n_i > \gd(K_iV)$ and $v$ vanishes when it moves one step along the $i$-th direction. In this situation, $v \in (K_i V)_\bfn$, so $K_iV$ has a direct summand supported on the hyperplane $\mathcal{H}$ consisting of objects $\bfr$ with $r_i = n_i$. Clearly, the zeroth homology group of this summand is also supported on $\mathcal{H}$, so the generation degree of this summand must be at least $n_i$. This forces $\gd(K_iV) \geqslant n_i$, which is a contradiction.

(6) To prove the first inequality, we assume that $t_i(V)$ is finite. If the inequality does not hold, then $t_i(\Sigma_j V) \geqslant 0$, so there exist an object $\bfn$ and a nonzero element $v \in (\Sigma_j V)_{\bfn}$ such that $n_i > t_i(V)$, and $v$ vanishes when it moves along the $i$-th direction. That is, one has $\C(\bfn, \bfn + \bfo_i) \cdot v = 0$ where $v$ is viewed as an element in $\Sigma_j V$. But $(\Sigma_j V)_{\bfn} = V_{\bfn + \bfo_j}$. By the definition of shift functors, one has $\rho_j (\C(\bfn, \bfn + \bfo_i)) \cdot v = 0$ where $v$ is regarded as an element in $V$. Note that
\begin{equation*}
\rho_j (\C(\bfn, \bfn + \bfo_i)) \subseteq \C(\bfn + \bfo_j, \bfn + \bfo_j + \bfo_i),
\end{equation*}
so $v \in V_{\bfn + \bfo_j}$ also vanishes when it moves one step along the $i$-th direction. Consequently, $t_i(V) \geqslant (\bfn + \bfo_j)_i \geqslant n_i$, which is a contradiction.

Now we turn to the second inequality. Since $t_i(V) \geqslant 0$, we know that $K_iV$ is nonzero. The second inequality can be proved similarly by noting that if $i = j$ and $n_i \geqslant t_i(V)$, then one can deduce that $t_i(V) \geqslant (\bfn + \bfo_i)_i = n_i + 1$, which is also a contradiction.
\end{proof}

\begin{remark} \normalfont
The above results have been verified when $\C$ is a skeleton of the category $\FI$ in \cite{CEFN, L, LY}. In particular, for $\FI$-modules $V$, one has $t(V) = \gd(KV)$. This fact plays a crucial role in \cite{L}, where the second author provided an alternative proof for the upper bound of Castelnuovo-Mumford regularity of $\FI$-modules appearing in \cite[Theorem A]{CE}. However, this equality in general does not hold when $\C$ is a skeleton of $\FI^m$ for $m > 1$, and the reader can easily find a counterexample.
\end{remark}

Recall that for a $\C$-module $V$, we let $t(V)$ be the sum of all $t_i(V)$ for $i \in [m]$, and $t(V) \geqslant -m$. The above results have two useful corollaries.

\begin{corollary} \label{corollary one}
Let $V$ be a $\C$-module and $i \in [m]$. If $t_i(V) \geqslant 0$, then $t(V/K_iV) \leqslant t(V) -1$.
\end{corollary}

\begin{proof}
Without loss of generality we can assume that $t(V)$ is finite, so $t_j(V)$ is also finite for every $j \in [m]$. Since $t_i(V) \geqslant 0$, we know that $K_i(V) $ is nonzero. Consider the short exact sequence
\begin{equation*}
0 \to V/K_iV \to \Sigma_i V \to D_iV \to 0
\end{equation*}
induced by the exact sequence
\begin{equation*}
0 \to K_i V \to V \to \Sigma_i V \to D_i V \to 0.
\end{equation*}
It is not hard to see that $t_j(V/K_iV) \leqslant t_j (\Sigma_i V)$ for all $j \in [m]$ by the definition of torsion vectors. But by the previous lemma, $t_j(\Sigma_i V) \leqslant t_j(V)$ for $j \neq i$, and $t_i(\Sigma_i V) \leqslant t_i(V) -1$. The conclusion follows.
\end{proof}

\begin{corollary} \label{corollary two}
Let $V$ be a $\C$-module presented in finite degrees and $i \in [m]$, and suppose that $\hd_2(D_i V)$ is finite. Then $K_i V$ is generated in finite degrees, $V/K_i V$ is presented in finite degrees, and $t_i(V)$ is finite. More precisely, $\gd(K_i V)$, $\prd(V/K_i V)$ and $t_i(V)$ are all less than or equal to $\max\{ \prd(V), \hd_2 (D_i V)\}$.
\end{corollary}

\begin{proof}
We may assume $V$ is nonzero, for otherwise the statements are trivial.
Since $V$ is presented in finite degrees, so are $\Sigma_i V$ and $D_i V$ by part (3) of Lemma \ref{basic properties}, where we let $S = \{ i \}$. Now break the exact sequence $0 \to K_i V \to V \to \Sigma_i V \to D_i V \to 0$ into two short exact sequences
\begin{equation*}
0 \to K_i V \to V \to V/K_i V \to 0
\end{equation*}
and
\begin{equation*}
0 \to V/K_i V \to \Sigma_i V \to D_i V \to 0.
\end{equation*}
We have $\gd(V/K_iV) \leqslant \gd(V)$. From the long exact sequence of homology groups associated to the second short exact sequence, we deduce that:
\begin{gather*}
\hd_1(V/K_i V) \leqslant \max\{ \hd_1(\Sigma_i V), \hd_2 (D_i V) \} \leqslant \max\{ \prd(V), \hd_2(D_i V) \},
\end{gather*}
where we used Lemma \ref{basic properties} (3). Hence, $V/K_i V$ is presented in finite degrees.

Similarly, from the long exact sequence associated to the first short exact sequence, we deduce that:
\[
\gd(K_i V) \leqslant \max\{ \gd(V), \hd_1(V/K_i V) \} \leqslant \max\{ \prd(V), \hd_2( D_i V) \}.
\]
Hence, $K_i V$ is generated in finite degrees. By Lemma \ref{basic properties} (5), it follows that
$t_i(V)\leqslant \max\{ \prd(V), \hd_2(D_i V)\}$.
\end{proof}

\section{A proof of the main theorem} \label{main proof}

Let $V$ be a $\C$-module. In this section we follow the algorithm described in \cite[Section 3]{GL3} to construct a tree of quotient modules of $V$, and use it to prove the main theorem.

An element $i \in [m]$ is said to be a \emph{singular index} of $V$ if $K_i V \neq 0$, and otherwise it is called a \emph{regular index} of $V$. Let $\mathcal{S}(V)$ and $\mathcal{R}(V)$ be, respectively, the sets of singular indices and regular indices of $V$. For every element $i \in \mathcal{S} (V)$, $K_iV$ is nonzero, so $V/K_iV$ is a proper quotient module of $V$.

\begin{definition} \label{children}
Let $V$ be any $\C$-module. We call the collection of quotient modules $V/K_i V$ for all $i\in \mathcal{S}(V)$ the \emph{children} of $V$.
\end{definition}

We now construct a tree as follows. First, we put $V$ on the zeroth level of the tree. Next, we put the children of $V$ on level -1 of the tree, and draw an arrow from $V$ to each of its children. We continue this process for each child of $V$ and keep going to get a tree $\mathcal{T}(V)$ of quotient modules of $V$. Each vertex in this tree (except $V$ itself) is called a \emph{descendant} of $V$. Clearly, for the $s$-th level on $\mathcal{T}(V)$, there are at most $m^{-s}$ vertices (note that $s$ is a non-positive integer). But for an arbitrary $\C$-module, $\mathcal{T}(V)$ might not be a finite tree since it may have infinitely many levels.

\begin{remark} \normalfont
Subsets of $\mathcal{R}(V)$ are called \emph{nil subsets} in \cite{GL3}, and $\mathcal{R}(V)$ is called the maximal nil subset.
\end{remark}

For any subset $S\subset [m]$, define
\[ F_S V
=\left( \bigoplus_{i\in S} D_i V \right) \oplus \left( \bigoplus_{i\in [m]\setminus S} \Sigma_i V \right),
\]
that is, $F_S V  = D_S V \oplus \Sigma_{[m]\setminus S} V $.
In particular, $F_\emptyset V = \Sigma_{[m]} V$ and $F_{[m]} V = D_{[m]} V$. For each $i\in [m]$, there is an exact sequence
\[ 0 \to V/K_i V \to \Sigma_i V \to D_i V \to 0. \]
It follows that if $S\subset T \subset [m]$, then there is a short exact sequence
\begin{equation} \label{short exact sequence for F}
0 \to \bigoplus_{i \in T\setminus S}  V/K_i V \to F_S V \to F_T V \to 0.
\end{equation}

\begin{lemma} \label{filtration}
Let $V$ be a $\C$-module. Then we have a short exact sequence
\begin{equation*}
0 \to \bigoplus_{i \in \mathcal{S}(V)} V/K_i V \to F_{\mathcal{R}(V)} V \to D_{[m]} V \to 0.
\end{equation*}
That is, the kernel of the surjective map $F_{\mathcal{R}(V)} V \to D_{[m]} V$ is the direct sum of the children of $V$.
\end{lemma}

\begin{proof}
Immediate from \eqref{short exact sequence for F} by taking $S=\mathcal{R}(V)$ and $T=[m]$.
\end{proof}

\begin{remark} \normalfont
We would like to point out the difference of notations between this paper and \cite{GL3} to avoid possible confusion. Specifically, for $S \subseteq [m]$, the module $D_SV$ in this paper is the cokernel of the map $V^S \to \Sigma_S V$, whereas $D_SV$ in \cite{GL3} is the cokernel of the map $V^S \to \Sigma_{[m]} V$. In other words, the module $D_S V$ in \cite{GL3} is same as $F_S V$ in this paper; in particular, for $S = \mathcal{R}(V)$, the module $D_{\mathcal{R}(V)}V$ defined in \cite{GL3} is precisely $F_{\mathcal{R}(V)} V$ in this paper.
\end{remark}

The following lemma is from \cite[Proposition 4.3]{GL3} (keeping in mind the differences between our present notations and \cite{GL3}.) For the convenience of the reader, we give its proof since the assumptions in \cite{GL3} are somewhat different.

\begin{lemma} \label{recursive inequality}
Let $V$ be a nonzero $\C$-module. If $S\subset \mathcal{R} (V)$, then $\reg(V) \leqslant \reg(F_S V)  + 1$. In particular, $\reg(V) \leqslant \reg(F_{\mathcal{R}(V)} V) + 1$.
\end{lemma}

\begin{proof}
We shall prove by induction on $s\geqslant 0$ that if $V$ is \emph{any} nonzero $\C$-module and $S\subset \mathcal{R} (V)$, then
\[ \hd_s(V) \leqslant \reg(F_S V) + s + 1. \]
This would imply that $\reg(V)\leqslant \reg(F_S V)+1$.

Let us first consider the case $s=0$. We have $\hd_0(V)=\gd(V) = \gd(D_{[m]}V) + 1$, where we used Lemma \ref{basic properties} (2). Since $D_{[m]} V$ is a quotient of $F_S V$, we have $\gd(D_{[m]} V) \leqslant \gd(F_S V)$. Hence,
\[\hd_0(V) \leqslant \gd(F_S V) + 1 \leqslant \reg(F_S V) + 1.\]

Next, suppose $s\geqslant 1$. Take a short exact sequence $0\to  W \to P \to V \to 0$ where $P$ is a free $\C$-module with $\gd(P)=\gd(V)$. If $W$ is zero, then $V$ is a free module and we are done, so suppose $W$ is nonzero. Since $P$ is free, we have $\mathcal{R}(P)=[m]$, so $\mathcal{R}(W)=[m]$, and so $S\subset \mathcal{R}(W)$. Therefore,
\[ \hd_s(V) \leqslant \hd_{s-1}(W) \leqslant \reg(F_S W) + s, \]
where the last inequality holds by induction hypothesis. It suffices to show that $\reg(F_S W) \leqslant \reg(F_S V) + 1$.

Since $\Sigma_S$ is an exact functor, we have the short exact sequence $0\to \Sigma_S W \to \Sigma_S P \to \Sigma_S V \to 0$ and a commuting diagram
\[
\xymatrix{
0 \ar[r] & W^S \ar[r] \ar[d] & P^S \ar[r] \ar[d] & V^S \ar[r] \ar[d] & 0\\
0 \ar[r] & \Sigma_{[m]} W \ar[r] & \Sigma_{[m]} P \ar[r] & \Sigma_{[m]} V \ar[r] & 0.}
\]
Hence we have an exact sequence $0 \to F_S W \to F_S P \to F_S V \to 0$. By Lemma \ref{basic properties} (1), $F_S P$ is projective and
\[ \gd(F_S P) \leqslant \gd(P) = \gd(V) = \gd(D_{[m]} V)+1 \leqslant \gd(F_S V)+1, \]
so $\gd(F_S W) \leqslant \max\{ \hd_1(F_S V), \gd(F_S P) \} \leqslant \max\{ \hd_1(F_S V), \hd_0(F_S V) + 1 \}$.
For each $r\geqslant 1$ we have
$\hd_r(F_S W) = \hd_{r+1} (F_S V)$.
Hence, $\reg(F_S W) \leqslant \reg(F_S V) + 1$.
\end{proof}

Now we are ready to prove the following main theorem.

\begin{theorem} \label{main theorem}
Let $V$ be a $\C$-module. Then $\reg(V)$ is finite if and only if $V$ is presented in finite degrees.
\end{theorem}

\begin{proof}
One direction is immediate since by the definition of regularity, if $\reg(V)$ is finite, so are all homological degrees of $V$. The proof of the other direction is based on an induction of $\gd(V)$.

Suppose that $V$ is presented in finite degrees. If $\gd(V) = -1$, that is, $V = 0$, the conclusion holds trivially.  Assume that $V$ is nonzero.

\textbf{Step 1:} We show that every vertex in the tree $\mathcal{T}(V)$ is presented in finite degrees. It suffices to show that if a nonzero vertex $W$ is presented in finite degrees, so are its children.
Suppose $i\in [m]$. By statements (2) and (3) of Lemma \ref{basic properties},
$D_i W$ is presented in finite degrees, and
\[ \gd(D_i W) \leqslant \gd(D_{[m]} W) = \gd(W) - 1 \leqslant \gd(V) - 1\]
since $W$ is a quotient module of $V$. Therefore, by the induction hypothesis, $\reg(D_i W)$ is finite, so in particular $\hd_2(D_i W)$ is finite; one has:
\[ \hd_2(D_iW) \leqslant \reg(D_i W) + 2.\]
Consequently, all children of $W$ are presented in finite degrees by Corollary \ref{corollary two}.

\textbf{Step 2:} We show that $\mathcal{T}(V)$ is a finite tree. Let $W$ be a nonzero vertex in $\mathcal{T}(V)$. We already proved that $W$ is presented in finite degrees in the previous step, and $\reg(D_{[m]} W)$ is finite by the induction hypothesis. By Corollary \ref{corollary two}, $t(W)$ is a finite number; by Corollary \ref{corollary one}, $t(W') < t(W)$ for any child $W'$ of $W$. Since $t(U) \geqslant -m$ for every vertex $U$ in $\mathcal{T}(V)$, it can has only finitely many levels. As each level has only finitely many vertices, the claim follows.

\textbf{Step 3:} We show that every vertex in the tree $\mathcal{T}(V)$ has finite regularity. Take an arbitrary vertex $W$ on the lowest level of $\mathcal{T}(V)$. Of course we can assume that $W$ is nonzero. Since $W$ has no children, we have $K_{[m]}W=0$, so we have a short exact sequence
\begin{equation*}
0 \to W^{\oplus m} \to \Sigma_{[m]} W \to D_{[m]} W \to 0.
\end{equation*}
The $\C$-module $D_{[m]}W$ is presented in finite degrees, and has finite regularity by the induction hypothesis. Moreover, $\mathcal{R}(W) = [m]$ and $F_{\mathcal{R}(W)} W= D_{[m]} W$. Thus by Lemma \ref{recursive inequality}, $\reg(W) \leqslant \reg (D_{[m]} W) + 1$ is finite. Consequently, every vertex on the lowest level has finite regularity.

Now we consider an arbitrary vertex $U$ on the second lowest level. Note that all children of $U$ have finite regularity as they appear in the lowest level, and $D_{[m]} U$ has finite regularity by the induction hypothesis. Combining the conclusions of Lemma \ref{filtration} and Lemma \ref{recursive inequality}, conclude that $U$ has finite regularity, too.

Since the tree $\mathcal{T}(V)$ has only finitely many levels, recursively one can show that all vertices in it, including $V$, have finite regularity.
\end{proof}

Applying Theorem \ref{main theorem} to a skeleton of $\FI^m$, we immediate deduce the finiteness of regularity of $\FI^m$-modules presented in finite degrees, and hence partially answered a question proposed in \cite[Subsection 6.1]{LY2}.

\begin{remark} \normalfont
As mentioned before, the above proof relies on an inductive machinery originated in \cite{GL2} and further developed in \cite{GL3}. In the second paper, for a few other categories such as $\FI_d$ and $\mathrm{OI}_d$ and Noetherian commutative coefficient rings, regularity of finitely generated representations is shown to be finite. But at this moment we cannot establish a version of Theorem \ref{main theorem} for these categories because numerical invariants for their representations, sharing similar properties as $t(V)$, are not available yet. Besides, for an $\FI^m$-module $V$ presented in finite degrees, an explicit upper bound of $\reg(V)$ is still missing for $m > 1$. In the special case $m = 1$, it is not hard to find an upper bound for $\reg(V_T)$, and use it to deduce an upper bound for $\reg(V)$, as was done in \cite{L}. However, for $m > 1$, we are not able to obtain a simple upper bound for the regularity of $V_T$.
\end{remark}

\section{A few consequences} \label{consequences}

\subsection{Category of modules presented in finite degrees}
A consequence Theorem \ref{main theorem} is:

\begin{corollary} \label{abelian category}
The category of $\C$-modules presented in finite degrees is abelian.
\end{corollary}

\begin{proof}
Let $\alpha: U \to V$ be a morphism such that both $U$ and $V$ are presented in finite degrees. We have to show that the kernel, cokernel, and image of $\alpha$ are all presented in finite degrees.

Consider the short exact sequence
\begin{equation*}
0 \to \im \alpha \to V \to \coker \alpha \to 0.
\end{equation*}
As a quotient module $U$, $\im \alpha$ is generated in finite degrees. Applying the homology functor we deduce that $\coker \alpha$ is presented in finite degrees. Therefore, both $\coker \alpha$ and $V$ have finite regularity by the previous theorem, and hence their homological degrees are finite. Consequently, all homological degrees of $\im \alpha$ are finite as well.

Now turn to the short exact sequence
\begin{equation*}
0 \to \ker \alpha \to U \to \im \alpha \to 0.
\end{equation*}
By a similar argument, we deduce that all homological degrees of $\ker \alpha$ are finite.
\end{proof}

\begin{remark} \normalfont
Let $V$ be a $\C$-module. Consider the following conditions:
\begin{enumerate}
\item $V$ is presented in finite degrees;
\item all homological degrees of $V$ are finite;
\item $\reg(V)$ is finite;
\item the category of $\C$-modules presented in finite degrees is abelian.
\end{enumerate}
Clearly, (3) implies (2), and (2) implies (1). Moreover, (1) and (2) are equivalent if and only if (4) holds. Indeed, if (1) implies (2), then the proof of the above corollary tells us that (4) holds. Conversely, suppose that (4) holds and $V$ is presented in finite degrees. Let $0 \to W \to P \to V \to 0$ be a short exact sequence such that $P$ is a free module generated in finite degrees (and so presented in finite degrees). Then $W$ is presented in finite degrees as well. In particular, $\hd_2(V) = \hd_1(W)$ is finite. Replacing $V$ by $W$ and repeating the above argument, one can eventually show that all homological degrees of $V$ are finite.
\end{remark}

\subsection{$\FI^m$-modules}
By Corollary \ref{abelian category}, the category of $\FI^m$-modules presented in finite degrees is an abelian category. Using this fact, many previously know results (for example in \cite{LY2}) about finitely generated $\FI^m$-modules over a commutative Noetherian coefficient ring, whose proofs only rely on the condition that both $\gd(V)$ and $\prd(V)$ are finite, can be extended to $\FI^m$-modules presented in finite degrees over any commutative coefficient ring. For example, \emph{relative projective $\FI^m$-modules} (also called \emph{semi-induced modules or $\sharp$-filtered modules} in literature) are defined in \cite[Subsection 1.4]{LY2} over any commutative coefficient ring. By \cite[Theorem 1.3]{LY2}, these modules are presented in finite degrees. Accordingly, we can extend \cite[Theorem 1.5]{LY2} to the setup of $\FI^m$-modules presented in finite degrees over any commutative ring.

\begin{theorem}
Let $V$ be an $\FI^m$-module presented in finite degrees over a commutative ring $k$. Then there exists a complex
\begin{equation*}
F^{\bullet}: \quad 0 \to V \to F^0 \to F^1 \to \ldots \to F^\ell \to 0
\end{equation*}
such that the following statements hold:
\begin{enumerate}
\item each $F^j$ is a relative projective module with $\gd(F^j) \leqslant \gd(V) - j$;
\item $\ell \leqslant \gd(V)$,
\item all homology groups $H^j (F^{\bullet})$ of this complex are torsion modules presented in finite degrees.
\end{enumerate}

Consequently, $\Sigma_1^{n_1} \ldots \Sigma_m^{n_m} V$ is a relative projective module if $n_i \geqslant t_i(H^j (F^{\bullet})) + 1$ for all $0 \leqslant j \leqslant \ell$ and $i \in [m]$.
\end{theorem}

\begin{proof}
The proof of \cite[Theorem 1.5]{LY2} as well as its prerequisite results, including \cite[Porposition 4.10, Lemma 4.8]{LY2}, only relies on the condition that both $\gd(V)$ and $t_i(V)$ for all $i \in [m]$ are finite. This condition still holds for $\FI^m$-modules presented in finite degrees over any commutative ring.
\end{proof}

\begin{remark} \normalfont
When $m = 1$, in \cite{LR, R2}, Ramos and the second author proved that the cohomology groups in the above finite complex are precisely the local cohomology groups of $\FI$-modules. Furthermore, Nagpal, Sam, and Snowden showed that the regularity of an $\FI$-module $V$ can be described in terms of degrees of these cohomology groups; see \cite[Conjecture 5.19]{LR} and \cite[Theorem 1.1]{NSS}. For $m > 1$, we expect these results still hold, though we could not establish them.
\end{remark}

\begin{remark} \normalfont
A generalization of the category $\FI$ is the category $\FI_G$, where $G$ is a (possibly infinite) group. This category encodes the wreath products of $G$ and all symmetric groups, and share very similar representation theoretic properties as $\FI$; see \cite{LR}. The above results for $\FI^m$-modules can be extended to $\FI_G^m$-modules using similar proofs.

\end{remark}

\end{document}